\documentclass[12pt]{article}
\usepackage{amsfonts}
\usepackage{latexsym}
\usepackage{color, amsmath,amssymb, amsfonts, amstext,amsthm, latexsym, dsfont}
\usepackage{epic,eepic}
\usepackage{graphicx}
\usepackage{longtable}
\usepackage{appendix}
\usepackage{graphicx}
\usepackage{float}
\usepackage[colorlinks=false, linktocpage=true]{hyperref}

\newtheorem{example}{Example}

\newtheorem{thm}{\bf Theorem} [section]
\newtheorem{lem}[thm]{Lemma}
\newtheorem{prop}[thm]{Proposition}
\newtheorem{rem}[thm]{Remark}
\numberwithin{equation}{section}
 \numberwithin{Lem}{section}
 \numberwithin{Defi}{section}
 \numberwithin{Theo}{section}
 \numberwithin{Rem}{section}
  \numberwithin{Coro}{section}
  \numberwithin{Fig}{section}

\title{Onsager-Machlup action functional for stochastic partial differential equations with L\'{e}vy noise
\footnotetext{\\
Jianyu Hu\\
School of Mathematics and Statistics \& Center for Mathematical Sciences \& Hubei National Center for Applied Mathematics, Huazhong University of Science and Technology, Wuhan
430074, China.\\
E-mail: jianyuhu@hust.edu.cn}
\footnotetext{\\Jinqiao Duan\\
Department of Applied Mathematics, Illinois Institute of Technology, Chicago, IL 60616,
USA.\\
E-mail: duan@iit.edu\\
$^*$ Corresponding author}}
\author{Jianyu Hu and Jinqiao Duan$^{*}$}

\begin{document}

\maketitle

\begin{abstract}
This work is devoted to deriving the Onsager-Machlup action functional for stochastic partial differential equations with (non-Gaussian) L\'{e}vy process as well as Gaussian Brownian motion. This is achieved by applying the Girsanov transformation for probability measures and then by a path representation. This enables the investigation of the most probable transition path for  infinite dimensional stochastic dynamical systems modeled by stochastic partial differential equations, by minimizing the Onsager-Machlup action functional.


\end{abstract}

\section{Introduction}
Stochastic dynamical systems are mathematical models for complex phenomena in physical, chemical and biological sciences \cite{arnold2013random,duan2015introduction,duan2014effective,imkeller2012stochastic}. In contrast to the deterministic systems, noisy fluctuations result in the possibility of transitions between two metastable states \cite{ditlevsen1999observation,qiu2000kuroshio}. The Onsager-Machlup action functional is an essential tool to study these transition phenomena. We will derive the Onsager-Machlup action functional for stochastic partial differential equations in Hilbert space with non-Gaussian L\'{e}vy process and this enables the study of the most probable transition pathway between two metastable states.

Onsager and Machlup considered the probability of paths of a diffusion process \cite{onsager1953fluctuations}. For stochastic differential equations, the Onsager-Machlup action functional has been widely investigated. A rigorous mathematical derivation was carried out by Ikeda and Watanabe \cite{ikeda2014stochastic}, Takahashi and Watanabe \cite{takahashi1981probability}, and Fujita and Kotani \cite{fujita1982onsager}. Shepp and Zeitouni \cite{shepp1992note} derived the Onsager-Machlup action functional in the Cameron-Martin space rather than twice continuously differentiable space. Capitaine \cite{capitaine1995onsager,capitaine2000onsager} derived the Onsager-Machlup action functional in the space of mean-square integrable functions. Hara and Takahashi \cite{hara2016stochastic} further derived Onsager-Machlup action functional for Brownian motion on a complete Riemannian manifold using a purely probabilistic method. The key point is to express the transition probability of a diffusion process in terms of a functional integral over paths of the process. Note that in the case of finite dimensional diffusions, the Onsager-Machlup action functional does not depend on the norm considered.

As in D\"{u}rr and Bach \cite{durr1978onsager}, Onsager-Machlup action functional may be regarded as a Lagrangian for determining the most probable path of a diffusion process by a variation principle. In quantum mechanics, Hochberg et.al \cite{hochberg1999effective} derived the effective action for stochastic partial differential equations using the path integral method. Cooper and Dawson \cite{cooper2016auxiliary} introduced an auxiliary field to obtain the effective action for the KPZ equation. This effective action is actually Onsager-Machlup action functional which is different from the classical action. Thus, the Onsager-Machlup action functional serves as a tool to study dynamical behaviors of stochastic dynamical systems, especially the most probable transition paths connecting metastable states. Using the Onsager-Machlup action functional, Yang et.al \cite{zheng2020maximum} studied the climate change for global warming. Huang et.al \cite{huang2020estimating} estimated the lower and upper bounds for the most probable transition time by minimizing the Onsager-Machlup action functional. In data science, it is possible to quantify most probable transition trajectories from the Onsager-Machlup action functional \cite{li2020machine,ren2020identifying}. There are other works studying the transition phenomena but using the Freidlin-Wentzell action functional via large deviation theory \cite{bouchet2014langevin,freidlin2012random,zhang2019transition}. We emphasize that large deviation theory only holds for small noise, but the Onsager-Machlup action functional is for any noise.

Note that most existing works on Onsager-Machlup action functional are for systems with Brownian motion. However, non-Gaussian random phenomena including heavy-tailed distributions have been found to be suitable in modeling biological evolution, climate change and other complex scientific and engineering systems \cite{bottcher2013levy,hao2014asymmetric,jourdain2012levy,zheng2016transitions}. As for deriving the Onsager-Machlup action functional for systems with non-Gaussion noise,  there are a few attempts. Bardina, Rovira and Tindel \cite{bardina2002asymptotic} dealt with the asymptotic evaluation of the Poisson measure for a tube. More recently, Chao and Duan \cite{chao2019onsager} used the Girsanov transformation to absorb the drift term and thus derive the Onsager-Machlup action functional in one dimensional systems with L\'{e}vy process.

For infinite dimensional stochastic systems, the results for Onsager-Machlup action functional are much fewer. The first result has been obtained by Dembo and Zeitouni\cite{dembo1991onsager} for trace class elliptic stochastic partial differential equations on a bounded domain of $\mathbb{R}^d$, and then by Mayer-Wolf and Zeitouni \cite{wolf1993onsager} in the non-trace case. Bardina, Rovira and Tindel\cite{bardina2003onsager} derived the Onsager-Machlup action functional for stochastic evolution equations in Hilbert space. But for infinite dimensional systems with L\'{e}vy process, as far as we know, the Onsager-Machlup action functional is not yet derived. This paper solves the problem for a class of infinite dimensional systems with L\'{e}vy process.

One difficulty in deriving Onsager-Machlup action functional for infinite dimensional systems is that the normalizing factor cannot be a function of the norm of the noise, unlike in finite dimensional systems. Moreover, the L\'{e}vy process is much more complex (Brownian motion is a special case of L\'{e}vy process) which is hard to handle. The methods in \cite{bardina2003onsager,dembo1991onsager,wolf1993onsager} encouter some new difficulties here, because the random variables represented by Karhunen-Lo\`{e}ve expansion are not independent. Inspired by the works \cite{chao2019onsager, durr1978onsager}, we are able to overcome these difficulties by a path representation and derive the Onsager-Machlup action functional for  stochastic partial differential equations with L\'{e}vy process as well as Brownian motion.

This paper is organized as follows. In section 2, we present the L\'{e}vy process and stochastic partial differential equation in Hilbert space and define the Onsager-Machlup action functional. In section 3, we apply the Girsanov transformation to obtain the relation between the two probability measures and then derive the Onsager-Machlup action functional by representing the  Radon-Nikodym derivative in terms of a path integral. We present an example in section 4, and finally end with some discussions in section 5.

\section{The framework}
We now recall some basic facts about L\'{e}vy process in Hilbert space, then introduce a class of stochastic partial differential equations \cite{applebaum2009levy,da2014stochastic,peszat2007stochastic}, and finally define the Onsager-Machlup action functional.
\subsection{L\'{e}vy process in Hilbert space}
Let $(\Omega,\mathcal{F},\mathbb{P})$ be a probability space with a filtration of right continuous  $\sigma$-algebras $(\mathcal{F}_t)_{t\geq0}$. All of the $\sigma$-algebras are assumed to be $\mathbb{P}$-completed. Given a real separable Hilbert space $H$, a L\'{e}vy process on $H$ can be defined by the formal series
\begin{equation}
\begin{split}
L(t)=\sum\limits_{j=1}^{\infty}L_j(t)e_j,
\end{split}
\end{equation}
where $\{e_j;j\geq1\}$ is an orthonormal basis in $H$, $\{L_j(t),t\in[0,1],j\geq1\}$ is a sequence of mutually independent L\'{e}vy processes in $\mathbb{R}$ adapted to $\mathcal{F}_t$. Assume that each $L_j(t)$ is a pure jump process which has the L\'{e}vy-Khinchin representation
\begin{align*}
\mathbb{E}e^{i\xi L_j(t)}=e^{-t\psi_j(\xi)},\ \ 0\leq t\leq1,\ \xi\in \mathbb{R},
\end{align*}
where
\begin{align*}
\psi_j(\xi)=\int_{\mathbb{R}}(1-e^{i\xi z}-\chi_{\{|z|<1\}}i\xi z)\nu_j(dz).
\end{align*}
Here, $\nu_j$ is the jump measure in the jth component. Then we have $\sum\limits_{j}\int_{\mathbb{R}}(|z|^2\wedge 1)\nu_j(dz)<\infty$. In this case, the series in (2.1) converges in probability, uniformly in $t$ on every compact interval $[0,1]$. Moreover, it can be expressed by the L\'{e}vy-It\^{o} decomposition
\begin{equation}
\begin{split}
L_j(t)=\int_{|z|<1}z \tilde{N}_j(t,dz)+\int_{|z|\geq1}z N_j(t,dz),
\end{split}
\end{equation}
where $N_j(dt,dz)$ is the Poisson measure associated with $L_j(t)$, and $\tilde{N}_j(dt,dz)=N_j(dt,dz)-\nu_j(dz)dt$ is the corresponding compensated Poisson random measure. In this paper, our L\'{e}vy process $L(t)$ has the following expression
\begin{equation}
\begin{split}
L(t)=\sum\limits_{j=1}^{\infty}L_j(t)=\sum\limits_{j=1}^{\infty}\int_{|z|<1}z\tilde{N}_j(t,dz)e_j+\sum\limits_{j=1}^{\infty}\int_{|z|\geq1}zN_j(t,dz)e_j.
\end{split}
\end{equation}

The following examples are of some important applications; see Cont and Tankov \cite{cont1975financial}.
\begin{example}(One sided exponentially tempered stable L\'{e}vy process)

As the L\'{e}vy jump measure $\nu_j$ of $L_j$, we take
\begin{equation}
\begin{split}
\nu_j(dr)=c_j\frac{e^{-\beta_jr}}{r^{1+\alpha_j}}\chi_{(0,\infty)}(r)dr.
\end{split}
\end{equation}
Then $L=\sum_jL_je_j$ is $H$-valued and square integrable if and only if $\sum_j\frac{c_j}{\beta_j^{2-\alpha_j}}<\infty$ for all $\alpha_j\in (0,2)$.
\end{example}
\begin{example}(Two sided exponentially tempered stable L\'{e}vy process)

As the L\'{e}vy jump measure $\nu_j$ of $L_j$, we take
\begin{equation}
\begin{split}
\nu_j(dr)=c_j^{-}\frac{e^{-\beta_j^{-}|r|}}{r^{1+\alpha_j^{-}}}\chi_{(\infty,0)}(r)dr+c_j^{+}\frac{e^{-\beta_j^{+}|r|}}{r^{1+\alpha_j^{+}}}\chi_{(0,\infty)}(r)dr.
\end{split}
\end{equation}
Then $L=\sum_jL_je_j$ is $H$-valued and square integrable if and only if $\sum_j(\frac{c_j^{-}}{(\beta_j^{-})^{2-\alpha_j^{-}}}+\frac{c_j^{+}}{(\beta_j^{+})^{2-\alpha_j^{+}}})<\infty$ for all $\alpha_j\in (0,2)$.
\end{example}

\subsection{Stochastic partial differential equations}
Consider the following stochastic partial differential equation in $H$:
\begin{equation}
\begin{split}
&dX(t)=[AX(t)+F(t,X(t))]dt+BdW(t)+ dL(t), t\in[0,1],\\
&X(0)=x \in H,
\end{split}
\end{equation}
where $A:D(A)\subset H\rightarrow H$ generates a $C_0$-semigroup $\{e^{tA};t\geq 0\}$, $F$ is a Lipschitz function defined on $[0,1]\times H$, $B$ is non-negative bounded linear operators from $H$ to $H$, $W(t)$ is a cylindrical Wiener process in $H$, and $L(t)$ is a L\'{e}vy process in $H$ defined as (2.3). We will also denote by $X^A$ the stochastic convolution of $A$ by $X$, that is, the solution to (2.1) with $F = 0$ and $x = 0$.

The norm in $H$ will be denoted by $|\cdot|_H$, and the scalar product by $\langle\cdot,\cdot\rangle$. The $L^2$ norm in $L^2([0,1];H)$ is denoted by $\|\cdot\|_2$. Let $\mathcal{L}(H)$ be the set of bounded linear operators on $H$. The norm $\|\cdot\|$ will be the usual operator norm defined on $\mathcal{L}(H)$. We shall make the following hypotheses:

$\mathbf{(H1)(Semigroup)}$ The operator A generates a self adjoint $C_0$-semigroup $\{e^{tA};t\geq0\}$ of negative type. Moreover, there exists a complete orthonormal system $\{e_j;j\geq1\}$ which diagonalizes $A$. We shall denote by $\{-\lambda_j;j\geq1\}$ the corresponding
set of eigenvalues and we assume that $\{\lambda_j;j\geq1\}$ is an increasing sequence of real numbers such that $\lambda_j>0$ and $\lim_{j\rightarrow\infty}\lambda_j=\infty$. Moreover, the semigroup $\{e^{tA};t\geq0\}$ is a contraction.

$\mathbf{(H2)(Lipschitz\ and\ Linear\ Growth)}$ There exists an integrable function $L_F:(0,1]\rightarrow(0,\infty)$ such that for all $s$ with $0<s\leq1$ and $x,y\in H$,
\begin{equation}
\begin{split}
&|e^{sA}(F(t,x)-F(t,y))|_H\leq L_F(s)|x-y|_H,\  t\in[0,1],\\
&|e^{sA}F(t,x)|_H\leq L_F(s)(1+|x|_H),\  t\in[0,1].
\end{split}
\end{equation}

$\mathbf{(H3)(Hilbert-Schmidt)}$ There exists an integrable function $L_B:(0,1]\rightarrow(0,\infty)$ such that for all $s$ with $0<s\leq1$ and $x,y\in H$,
\begin{equation}
\begin{split}
&|e^{sA}B(x-y)|_{HS}\leq L_B(s)|x-y|_H,\  t\in[0,1],\\
&|e^{sA}Bx|_{HS}\leq L_B(s)(1+|x|_H),\  t\in[0,1].
\end{split}
\end{equation}

$\mathbf{(H4)(Moment)}$ The L\'{e}vy jump measure satisfies
\begin{equation}
\begin{split}
\sum\limits_{j=1}^{\infty}\int_{|z|<1}|z|\nu_j(dz)<\infty.
\end{split}
\end{equation}

\begin{rem}
The contraction of the $C_0$-semigroup $e^{tA}$ guarantees that  the solution of (2.6) has a c\`{a}dl\`{a}g version. The hypotheses $\mathbf{H2}$ and $\mathbf{H3}$ are usual conditions which ensure the existence and uniqueness of a solution to (2.6). There are many works on the existence and uniqueness of a solution to (2.6), as in \cite{brzezniak2014strong,peszat2007stochastic,qiao2017path,sun2019pathwise}
\end{rem}

\begin{rem}
The hypothesis $\mathbf{H4}$ is very important. On the one hand, it ensures the availability of  Girsanov's transformation. On the other hand, it guarantees that the process $\int_0^t\int_{|y|_H<1}\exp((t-s)A)y \tilde{N}(dt,dy)+\int_0^t\int_{|y|_H\geq1}\exp((t-s)A)y N(dt,dy)$ is finite variation. These will be shown latter.
\end{rem}

Under the  hypotheses $\mathbf{H1,H2}$ and $\mathbf{H3}$, there exists a unique mild solution to (2.6). We say $X=\{X(t),t\in[0,1]\}$ is a mild solution to (2.6) if it is an $H$-valued $\mathcal{F}_t$ -adapted square integrable process such that
\begin{equation}
\begin{split}
X(t)&=\exp(tA)x+\int_0^t\exp((t-s)A)F(s,X(s))ds+\int_0^t\exp((t-s)A)BdW(s)\\
&+\int_0^t\int_{|y|_H<1}\exp((t-s)A)y\tilde{N}(ds,dy)+\int_0^t\int_{|y|_H\geq1}\exp((t-s)A)yN(ds,dy),
\end{split}
\end{equation}
for all $t\in[0,1]$.

The following proposition is a special case of \cite[Theorem 9.9]{peszat2007stochastic}.
\begin{prop}
Suppose that $\mathbf{H1,H2}$ and $\mathbf{H3}$ are satisfied. Then there exists a unique solution X(t) to (2.8) such that $X\in L^2(\Omega\times[0,1];H)$. Furthermore, the solution has a c\`{a}dl\`{a}g version.
\end{prop}
Note that the solution process $X^A$ is given by
\begin{equation}
\begin{split}
X^A(t)&=\int_0^t\exp((t-s)A)BdW(s)+\int_0^t\int_{|y|_H<1}\exp((t-s)A)y\tilde{N}(ds,dy)\\
&\ \ +\int_0^t\int_{|y|_H\geq1}\exp((t-s)A)yN(ds,dy),
\end{split}
\end{equation}

As shown in Proposition 2.3, this solution process has a c\`{a}dl\`{a}g version. Moreover, under the hypothesis $\mathbf{H4}$, we have the following lemma.
\begin{lem}
Let $Z(t)=\int_0^t\int_{|y|_H<1}\exp((t-s)A)y\tilde{N}(ds,dy)+\int_0^t\int_{|y|_H\geq1}\exp((t-s)A)yN(ds,dy)$, for $t\in[0,T]$. Then $Z(t)$ is finite variation process if and only if $\sum\limits_{j=1}^{\infty}\int_{|z|<1}|z|\nu_j(dz)<\infty$.
\end{lem}
\begin{proof}
Let $\Delta Z(t)=Z(t)-Z(t-)$, for $t\in[0,T]$. Then we have
\begin{align*}
\mathbb{E}\exp\{-\sum\limits_{0<t\leq T}|\Delta Z(t)|_H\}&=\mathbb{E}\exp\{-\sum\limits_{j=1}^{\infty}\int_0^T\int_{\mathbb{R}}|z|e^{-\lambda_j(T-t)}N_j(dt,dz)\},\\
&=\exp\{-\sum\limits_{j=1}^{\infty}\frac{1}{\lambda_j}(1-e^{-\lambda_jT})\int_{\mathbb{R}}(1-e^{|z|})\nu_j(dz)\}.
\end{align*}

Note that $1-e^{-\lambda_jT}\leq\lambda_j T$. Thus, if $\sum\limits_{j=1}^{\infty}\int_{\mathbb{R}}(1-e^{|z|})\nu_j(dz)\}=\infty$, or equivalently if $\sum\limits_{j=1}^{\infty}\int_{|z|<1}|z|\nu_j(dz)=\infty$, then $\sum\limits_{0<t\leq T}|\Delta Z(t)|_H=\infty$, $\mathbb{P}$-a.s. And if $\sum\limits_{j=1}^{\infty}\int_{\mathbb{R}}(1-e^{|z|})\nu_j(dz)\}<\infty$, or equivalently if $\sum\limits_{j=1}^{\infty}\int_{|z|<1}|z|\nu_j(dz)<\infty$, then
\begin{align*}
\mathbb{E}\sum\limits_{0<t\leq T}|\Delta Z(t)|_H\chi_{\{|Z(t)|_H\leq M\}}=\sum\limits_{j=1}^{\infty}\frac{1}{\lambda_j}(1-e^{-\lambda_jT})\int_{|z|<1}|z|\nu_j(dz)<\infty.
\end{align*}
Thus $Z(t)$ is of bounded variation. $\square$
\end{proof}

This lemma is very important, which is able to control the path integral while deriving the Onsager-Machlup action functional.

\subsection{Definition of Onsager-Machlup action functional}

Now, we define the Onsager-Machlup action functional \cite{bardina2003onsager}  associated with the stochastic partial differential equation (2.6). Our goal is to study the limiting behaviour of ratios of the form
\begin{equation}
\begin{split}
\gamma_{\epsilon}(\phi)=\frac{\mathbb{P}(\|X-\phi\|\leq\epsilon)}{\mathbb{P}(\|X^A\|\leq\epsilon)},
\end{split}
\end{equation}
when $\epsilon$ tends to $0$. Here $\phi$ is a deterministic function satisfying some regularity conditions and $\|\cdot\|$ is a suitable norm defined on the functions from $[0, 1]$ to $H$.

When $\lim\limits_{\epsilon\rightarrow0}\gamma_{\epsilon}(\phi)=\exp\{J_0(\phi)\}$ for all $\phi$ in a reasonable class of functions, the functional $J_0$ is called the Onsager-Machlup action functional associated to system (2.6). We will chose to work with the Hilbert norm on $L^2([0,1];H)$ for reasons in \cite{bardina2003onsager}.

As we can see, the Onsager-Machlup action functional quantifies the probability of the solution path up to a given function on a small tube. If we minimize this functional given two metastable states, we will obtain the most probable transition pathway between two fixed points \cite{ren2020identifying,zheng2020maximum}. This is an important application to study the dynamics for stochastic partial differential equations.

\section{Onsager-Machlup action functional}
In this section, we derive the Onsager-Machlup action functional for the stochastic partial differential equation (2.6). This result covers the results of both infinite case \cite{bardina2003onsager,capitaine1995onsager} and finite case \cite{chao2019onsager}. At first, in section 3.1, we apply the Girsanov transformation in order to reduce our problem to the evaluation of a functional of the stochastic convolution $X^A$. Then in section 3.2, we use the It\^{o} formula to express the Radon-Nikodym derivative in terms of a path integral. Thus in section 3.3, we obtain the Onsager-Machlup action functional by analyzing the limit behaviors of a path integral.

\subsection{An application of Girsanov's transformation}
In this subsection, we will apply the Girsanov transformation to obtain the relation between two probability measures. Actually, the Girsanov theorem for L\'{e}vy process does not work here, because after the transformation, the distribution of resulting noise in new probability space differs to the original noise with respect to the original probability space, which results in the failure of weak uniqueness. Benefiting from the independence of Brownian motion and Poisson random measure, we can use the Brownian motion to absorb the drift. This idea comes from the finite case in \cite{chao2019onsager}.

To do so, we first introduce an auxiliary system. We will take our deterministic function $\phi$ in a so-called Cameron-Martin space. Fix a function $h\in L^2([0,1];H)$. Let $\phi^h$ the solution of the infinite dimensional equation
\begin{equation}
\begin{split}
&d\phi^h(t)=A\phi^h(t)dt+Bh(t)dt, t\in[0,1],\\
&\phi^h(0)=x.
\end{split}
\end{equation}

The assumption $h\in L^2([0,1];H)$ is required in order to apply Girsanov's transformation. We also assume that for each $t\in[0,1]$, $F(t,x)\in Im(B)$ a.s, and
\begin{equation}
\begin{split}
\mathbb{E}[\exp(\frac{1}{2}\int_0^1|B^{-1}F(t,X(t)|_H^2)]<\infty.
\end{split}
\end{equation}

Recall that we will call measures $\mu_X$ and $\mu_Y$ equivalent ($\mu_X\sim \mu_Y$) if $\mu_X$ is absolutely continuous with respect to $\mu_Y$ ($\mu_X\ll\mu_Y$) and $\mu_Y\ll\mu_X$. Applying the Girsanov theorem, we have the following lemma.
\begin{lem}
Let $X_t$ and $Y_t$ be two solution processes defined by the stochastic partial differential equations with respect $(\Omega, \mathcal{F}, (\mathcal{F}_t)_{t\geq0}, \mathbb{P})$
\begin{equation}
\begin{split}
dX(t)&=[AX(t)+F(t,X(t))]dt+BdW(t)+ dL(t),\\
dY(t)&=[AY(t)+G(t,Y(t))]dt+BdW(t)+ dL(t),
\end{split}
\end{equation}
where the driven L\'{e}vy process is defined as (2.3) satisfying the hypothesis $\mathbf{H4}$. Here, $X_0=Y_0=x\in H$, $t\in[0,1]$, $A,B,F$ satisfy the previous condition, and $G$ is required the same condition as $F$. Then, we have $\mu_X\sim\mu_Y$ and the Radon-Nikodym derivative of $\mu_X$ with respect to $\mu_Y$ is given by
\begin{equation}
\begin{split}
\frac{d\mu_X}{d\mu_Y}[Y_t(\omega)]&=\exp\{ \int_0^t\langle B^{-1}(F(s,Y(s))-G(s,Y(s))),dW(s)\rangle\\
& \ \ \ \ \ -\frac{1}{2}\int_0^t|B^{-1}(F(s,Y(s))-G(s,Y(s)))|_H^2ds\}.
\end{split}
\end{equation}
\end{lem}
\begin{proof}
Define
\begin{align*}
M_t=\exp&\{ \int_0^1\langle B^{-1}(F(s,Y(s))-G(s,Y(s))),dW(s)\rangle\\
\ \ \  &-\frac{1}{2}\int_0^1|B^{-1}(F(s,Y(s))-G(s,Y(s)))|_H^2ds\}.
\end{align*}
According to the Girsanov transformation \cite[Theorem 3.17]{jacod2013limit}\cite[Theorem 2.1]{qiao2017path}, $\widehat{\mathbb{P}}$ defined as $d\widehat{\mathbb{P}}(\omega)=M_t(\omega)d\mathbb{P}(\omega)$ is a probability measure, and the process $\widehat{W}_t:=W_t+\int_0^tB^{-1}(F(s,Y(s))-G(s,Y(s)))ds$ is a cylindrical Brownian motion under the new filtered probability space $(\Omega, \mathcal{F}, (\mathcal{F}_t)_{t\geq0}, \widehat{\mathbb{P}})$. Moreover, $\widetilde{N}(dt,dy)$ is still the compensated martingale measure with jump measure $\nu$ with respect to $\widehat{\mathbb{P}}$. Thus the process $Y_t$ respect to $\widehat{\mathbb{P}}$ has the stochastic differential representation
\begin{equation}
\begin{split}
dY(t)=[AY(t)+F(t,Y(t))]dt+Bd\widehat{W}(t)+ dL(t).
\end{split}
\end{equation}
Due to the uniqueness in distribution, we have the equality of $\mu_X$ and $\mu_Y^{\widehat{\mathbb{P}}}$. Then we obtain
\begin{align*}
\frac{d\mu_X}{d\mu_Y}[Y_t(\omega)]&=\frac{d\widehat{\mathbb{P}}}{d\mathbb{P}}(\omega)\\
&=\exp\{ \int_0^t\langle B^{-1}(F(s,Y(s))-G(s,Y(s))),dW(s)\rangle\\
& \ \ \ \ \ -\frac{1}{2}\int_0^t|B^{-1}(F(s,Y(s))-G(s,Y(s)))|_H^2ds\}.
\end{align*}
The proof is complete.  $\square$
\end{proof}

Now, let us consider the function $\gamma_{\epsilon}(\phi)$ defined as (2.10). For every $\epsilon>0$, applying lemma 3.1 with $G(s)=Bh(s)$ and noting that $Y(t)=\phi^h(t)+X^A(t)$, we have
\begin{equation}
\begin{split}
\mathbb{P}(\|X-\phi^h\|_2\leq\epsilon)=\widehat{\mathbb{P}}(\|Y-\phi^h\|_2\leq\epsilon)=\mathbb{E}[\exp\{\Lambda\}\chi_{\|X^A\|_2\leq\epsilon}],
\end{split}
\end{equation}
where
\begin{equation}
\begin{split}
\Lambda&=\int_0^1\langle B^{-1}F(s,X^A(s_{-})+\phi^h(s))-h(s),dW(s)\rangle\\
&\ \ -\frac{1}{2}\int_0^1|B^{-1}F(s,X^A(s_{-})+\phi^h(s))-h(s)|_H^2ds.
\end{split}
\end{equation}

\begin{rem}Because we use the Girsanov theorem of Brownian motion to absorb the drift term by renorming the probability space, the only difference to Brownian motion is the solution $X^A$ of linear system (F=0 and x=0) which is now
\begin{equation}
\begin{split}
X^A(t)&=\int_0^te^{(t-s)A}BdW(s)+\int_0^t\int_{|y|_H<1}e^{(t-s)A}y\widetilde{N}(dt,dy)\\
&+\int_0^t\int_{|y|_H\geq1}e^{(t-s)A}yN(dt,dy).
\end{split}
\end{equation}
But this is just the essential difference to the Brownian case. Bardina, Rovira and Tindel \cite{bardina2003onsager} use the $L^2$ techniques to analyze the limiting behaviors through representing the process by random variables. But now, this method will meet new difficulties, because the representation of the process $X^A$ by Karhunen-Lo\`{e}ve expansion can't obtain the independence of random variables. Inspired of by the works\cite{chao2019onsager,durr1978onsager}, we could use It\^{o}'s formula to express the Radon-Nikodym derivative in terms of a path integral.
\end{rem}

\begin{rem}By weak uniqueness, the transformed $Y(t)=\phi^h(t)+X^A(t)$ is the solution of the following SPDEs
\begin{equation}
dY(t)=[AY(t)+Bh(t)]dt+BdW(t)+dL(t)
\end{equation}
We will apply It\^{o}'s formula to this process in next subsection. Furthermore, the norm of $X^A(t)$ is less than $\epsilon$, thus $y(t)$ can be considered as a perturbation of $\phi^h(t)$. Then we could take the Taylor's expansion at $\phi^h(t)$ for every $C^2$ function.
\end{rem}

\subsection{Representation by a path integral}
In this subsection, we will use It\^{o}'s formula to represent the Radon-Nikodym derivative in terms of a path integral. This method works because the Poisson stochastic integral can also be expressed by path integral and it can be controlled by bounded variation. Thus, we are able to derive the Onsager-Machlup action functional.

Let $V:H\rightarrow \mathbb{R} $ be $C^2$ function. Assume that for each $x$, $DV(x)=(B^{-1})^{*}(B^{-1}F(x)-h)$, $D^2V(x)\in L_{(HS)}(H,H)$ and that the mapping $x\mapsto D^2V(x)$ is uniformly continuous on any bounded subset of $H$. Applying It\^{o}'s formula \cite[Appendix 4]{peszat2007stochastic} \cite[Theorem 27.1]{metivier1982semimartingales} to $Y(t)$ driven by (3.9) in Hilbert space, we have
\begin{equation}
\begin{split}
V(y(t))-V(y(0))&=\int_0^t\langle B^{-1}F(s,y(s_{-}))-h(s),B^{-1}(Ay(s_{-})+Bh(s)\rangle ds\\
&+\int_0^t\langle B^{-1}F(s,y(s_{-}))-h(s),dW(s)\rangle\\
&+\frac{1}{2}\int_0^t\sum\limits_{j=1}^{\infty}D^2_{jj}V(y(s_{-}))Be_j\otimes Be_jds \\
&+ \sum\limits_{0\leq s\leq t}[V(y(s))-V(y(s_{-}))]\\
&+\int_0^t\int_{|\xi|_H<1}\langle B^{-1}F(s,y(s_{-}))-h(s),B^{-1}\xi\rangle\nu(d\xi)ds.
\end{split}
\end{equation}
Then the exponent part of Radon-Nikodym derivative can be represented as
\begin{equation}
\begin{split}
\Lambda&=V(y(1))-V(y(0))+ \sum\limits_{0\leq t\leq 1}[V(y(t))-V(y(t_{-}))]\\
&-\int_0^1b(y(t_{-})) dt+\int_0^1\int_{|\xi|_H<1}\langle B^{-1}F(s,y(t_{-}))-h(t),B^{-1}\xi\rangle\nu(d\xi)dt,
\end{split}
\end{equation}
where $b$ is defined as
\begin{equation}
\begin{split}
b(y(t_{-})))&=|B^{-1}F(t,y(t_{-})))-h(s)|_H^2+\sum\limits_{j=1}^{\infty}D^2_{jj}V(y(t_{-})))Be_j\otimes Be_j \\
&+2\langle B^{-1}F(t,y(t_{-})))-h(t),B^{-1}(Ay(t_{-}))+Bh(t)\rangle.
\end{split}
\end{equation}

We see that $\Lambda$ is now the functional of the path $y(t)$. But we can not always find the the primitive function $V$ given a function F. We will apply the Poincar\'{e} lemma to prove the the exitance of primitive function $V$. Actually in one dimension \cite{chao2019onsager,durr1978onsager}, the primitive function $V$ can be defined as a potential function
\begin{equation}
\begin{split}
V(x)=\frac{1}{B}\int^x[\frac{1}{B}F(y)-h]dy.
\end{split}
\end{equation}
But in infinite dimensional case or even two dimensional case, this potential is not always available. We have to use the Poincar\'{e} lemma which states the relation between closed form and exact form. It provides a sufficient condition for the existence of a potential function. Hence, we have the following lemma.

\begin{lem}
Let H be a separable Hilbert space with basis $\{e_j\}_{j\geq 1}$ and $F$ be a mapping from $H$ to $H$. For each $x\in H$, assume $DF(x)$ is symmetric operator from $H$ to $H$. Then there exists a smooth function $V:H\rightarrow \mathbb{R}$, such that for all $x\in H$, DV(x)=F(x).
\end{lem}
\begin{proof}
The dual space of Hilbert space is itself. So for each $x\in H$, $F(x)$ could be considered as a bounded linear functional in $H$. Thus, it defines a 1-form as
\begin{equation}
w=\langle F(x),de\rangle=\sum\limits_{j}F_j(x)de_j.
\end{equation}
Set $G_{ij}(x)=\langle DF(x)e_i,e_j\rangle$. Since $DF(x)$ is symmetric, i.e. $\langle DF(x)e_i,e_j\rangle=\langle e_i,DF(x)e_j\rangle$, we have $G_{ij}(x)=G_{ji}(x)$. Hence, the 2-form $dw$ is
\begin{equation}
\begin{split}
dw=\sum\limits_{ij}\langle e_i,DF(x)e_j\rangle de_i\wedge de_j=\sum\limits_{ij}G_{ij}(x) de_i\wedge de_j,
\end{split}
\end{equation}
where $de_i\wedge de_j$ is the natural wedge product which is bilinear and antisymmetric. Thus, we obtain
\begin{equation}
\begin{split}
dw=\sum\limits_{1\leq i<j}(G_{ij}(x)-G_{ji}(x)) de_i\wedge de_j=0.
\end{split}
\end{equation}
This implies that the mapping F is a closed 1-form. Then by the Poincar\'{e} lemma \cite[Page 137 Theorem 4.1]{lang2012fundamentals} \cite[Page 350 Theorem 33.20]{kriegl1997convenient}, it is an exact form, i.e. there exists a smooth function $V:H\rightarrow \mathbb{R}$, such that for all $x\in H$, $DV(x)=G(x)$. $\square$
\end{proof}

\begin{rem}
Note that our method will additionally need the symmetry of $DF(x)$, while $L^2$ technique doesn't. Bardina, Rovira and Tindel \cite[Proposition 3.7]{bardina2003onsager} uses a symmetrization technique by $S_{QR^{*}}=\frac{1}{2}(QR^{*}+(QR^{*})^{*})$ in the analysis of random variables. But they also need to assume that the symmetry of $\nabla_x F$, that is, $\nabla_x F$ can be diagonalized in the same basis as the operators A and B, when they further express the operator $S_{QR^{*}}$ by $\nabla_x F$ and $\phi^h$. That is, if for every $s\in[0,1]$, the gradient $\nabla_x F(s,\phi^h(s))$ is a trace class operator, they obtain
\begin{equation}
\begin{split}
S_{QR^{*}}=\frac{1}{2}\int_0^1Tr[\nabla_xF(s,\phi^h(s))]ds.
\end{split}
\end{equation}
Actually, we will see that the trace part in Onsager-Machlup action functional comes from the It\^{o} correction in stochastic analysis.
\end{rem}

\subsection{Derivation of Onsager-Machlup action functional}
Now, we are ready to derive the Onsager-Machlup action functional for the stochastic partial differential equation
\begin{equation}
\begin{split}
dX(t)=[AX(t)+F(t,X(t))]dt+BdW(t)+dL(t),
\end{split}
\end{equation}
with initial data $X(0)=x\in H$.

Our main result for this stochastic partial differential equation is in the following theorem.

\begin{thm}$\mathbf{(Onsager-Machlup\ action\ functional)}$
Assume that $\mathbf{H1,H2,H3}$ and $\mathbf{H4}$ are satisfied, $h\in L^2([0,1])$, $\phi^h$ is defined by (3.1) and $B^{-1}F$ is $C_b^2$ in $x$ uniformly in $s\in[0,1]$. If  for every $s\in[0,1]$, the gradient $\nabla_x F(s,\phi^h(s))$ is symmetric and trace class, and $\int_0^1Tr[\nabla_xF(s,\phi^h(s))]ds<\infty$, then the Onsager-Machlup action functional of system (3.18) is given by
\begin{equation}
\begin{split}
J_0(\phi^h,\dot{\phi}^h)&=-\frac{1}{2}\int_0^1|B^{-1}[(A\phi^h(t)+F(t,\phi^h(t))-\dot{\phi}^h(t)-\eta]|_H^2dt\\
&\ \ \ -\frac{1}{2}\int_0^1Tr[\nabla_xF(s,\phi^h(s))]ds,
\end{split}
\end{equation}
where $\eta=\int_{|\xi|_H<1}\xi\nu(d\xi)$.
\end{thm}

\begin{rem}
Up to a constant, the Onsager-Machlup action functional can also be written as
\begin{equation}
\begin{split}
J_0(\phi^h,\dot{\phi}^h)&=-\frac{1}{2}\int_0^1|B^{-1}[(A\phi^h(t)+F(t,\phi^h(t))-\dot{\phi}^h(t)]|_H^2dt\\
&\ \ \ -\frac{1}{2}\int_0^1Tr[\nabla_xF(s,\phi^h(s))]ds\\
&\ \ \ -\frac{1}{2}\int_0^1\langle B^{-1}[(A\phi^h(t)+F(t,\phi^h(t))-\dot{\phi}^h(t)],B^{-1}\int_{|\xi|_H<1}\xi\nu(d\xi)\rangle ds.
\end{split}
\end{equation}
As we can see, the first term is the main term, while the second term comes from the It\^{o} correction of Brownian motion. The L\'{e}vy noise results in the third term.

It is interesting to see that only small jumps contribute to the Onsager-Machlup action functional. Moreover, the effect is similar to adding the mean value of small jumps to the drift. Furthermore, if the integral $\eta=\int_{|\xi|_H<1}\xi\nu(d\xi)$ is the zero element in $H$, then the third term will vanish, in this case, the L\'{e}vy noise will have no effect to the Onsager-Machlup action functional.
\end{rem}
\begin{rem}We derived the Onsager-Machlup action functional for the stochastic partial differential equation
\begin{equation}
\begin{split}
dX(t)=[AX(t)+F(t,X(t))]dt+BdW(t)+dL(t),
\end{split}
\end{equation}
with initial data $X(0)=x\in H$. If this system does not have L\'{e}vy process $L(t)$(i.e., $\nu=0$),  then the Onsager-Machlup action functional will reduce to that in the Brownian case \cite{bardina2003onsager}.

Moreover, our results hold for the L\'{e}vy process with both small and large jumps, while Chao and Duan \cite{chao2019onsager} only dealt with small jumps. Our derivation is available for two reasons. One is that the Girsanov theorem still holds for large jumps. The other is that we can use the bounded variation to control the summation of the jumps.
\end{rem}

\begin{rem}The proof remains the same for a function $F$ satisfying

1) $F:[0,1]\times L^2([0,1];H)\rightarrow L^2([0,1];H)$.

2) $B^{-1}F(t,\cdot):L^2([0,1];H)\rightarrow L^2([0,1];H)$ is $C_b^2$ uniformly in $t\in[0,1]$.

3) For all $t\in[0,1]$ and $\xi\in L^2([0,1];H)$, $F(t,\xi)=F(t,\xi1_{[0,1]})$.
\end{rem}
Now, we prove Theorem 3.6.
\begin{proof}
As calculated in (3.11) and (3.12), the exponent part of Radon-Nikodym derivative can be represented as
\begin{equation}
\begin{split}
\Lambda&=V(y(1))-V(y(0))+ \sum\limits_{0\leq t\leq 1}[V(y(t))-V(y(t_{-}))]\\
&-\int_0^1b(y(t_{-})) dt+\int_0^1\int_{|\xi|_H<1}\langle B^{-1}F(s,y(t_{-}))-h(t),B^{-1}\xi\rangle\nu(d\xi)dt,
\end{split}
\end{equation}
where $b$ is defined as
\begin{equation}
\begin{split}
b(y(t_{-})))&=|B^{-1}F(t,y(t_{-})))-h(s)|_H^2+\sum\limits_{J=1}^{\infty}D^2_{jj}V(y(t_{-})))Be_j\otimes Be_j \\
&+2\langle B^{-1}F(t,y(t_{-})))-h(t),B^{-1}(Ay(t_{-}))+Bh(t)\rangle.
\end{split}
\end{equation}

The integrals with respect to time variable $t$ in (3.23) are Riemann integrals. Now we expand the exponent of (3.23) into a Taylor series around $y(t)=\phi^h(t)$ due to $|X^A(t)|\leq\epsilon$, and split the terms of zero order. If we choose $\epsilon$ small enough, the remaining terms can be made arbitrarily small.

Since $B^{-1}F$ is a $C_b^2$ in $x$ uniformly in $t\in[0,1]$, and $V$ is at least $C^2$ function, we obtain
\begin{equation}
\begin{split}
V(y(1))&=V(\phi^h(1))+\langle DV(\phi^h(1)),X^A(1)\rangle+o(\epsilon),\\
b(y(t-))&=|B^{-1}F(t,\phi^h(t))-h(t)|^2+2\langle B^{-1}F(s,\phi^h(t))-h(t),B^{-1}(A\phi^h(t)+Bh(t))\rangle\\
&\ \ +Tr(D_xF(t,\phi^h(t)))+Tr(\langle D_x^2F(t,\phi^h(t)),X^A(t)\rangle)\\
&\ \  +2\langle B^{-1}F(s,\phi^h(t))-h(t),B^{-1}AX^A(t)\rangle\\
&\ \ +2\langle B^{-1}F(s,\phi^h(t))-h(t),\langle D_xF(t,\phi^h(t)),X^A(t)\rangle\rangle +o(\epsilon),\\
\end{split}
\end{equation}
and
\begin{equation}
\begin{split}
\langle B^{-1}F(s,&y(t-))-h(t),B^{-1}\eta\rangle=\langle
B^{-1}F(s,\phi^h(t))-h(t),B^{-1}\eta\rangle\\
&\ \ \ \ \ \ \ \ \ \ \ \ \ \ \ \ \ \ \ \ \ \ \ \ \ \ \ \ +\langle B^{-1}\langle D_xF(s,\phi^h(t)),X^A(t)\rangle,B^{-1}\eta\rangle+o(\epsilon),
\end{split}
\end{equation}
where $\eta=\int_{|\xi|_H<1}\xi\nu(d\xi)$. Note that
\begin{equation}
\begin{split}
V(\phi^h(1))-V(\phi^h(o))=\int_0^1\langle\dot{\phi}^h(t),(B^{-1})^{*}(B^{-1}F(\phi^h(t),t)-h(t))\rangle dt.
\end{split}
\end{equation}

Since $|X^A(t)|\leq\epsilon$, by H\"{o}lder inequality and by the fact of $\dot{\phi}^h(t)=A\phi^h(t)-Bh(t)$, we obtain
\begin{equation}
\begin{split}
\Lambda&=V(\phi^h(1))-V(\phi^h(1))\\
&\ \ -\int_0^1Tr(D_xF(t,\phi^h(t)))dt-\int_0^1|B^{-1}F(t,\phi^h(t))-h(t)|^2dt\\
&\ \ -2\int_0^1\langle B^{-1}F(s,\phi^h(t))-h(t),B^{-1}\eta\rangle dt\\
&\ \ +\sum\limits_{0\leq t\leq 1}[V(y(t))-V(y(t-))]+O(\epsilon)\\
&=-\frac{1}{2}\int_0^1|B^{-1}[(A\phi^h(t)+F(t,\phi^h(t))-\dot{\phi}^h(t)]|_H^2dt\\
&\ \ \ -\frac{1}{2}\int_0^1Tr[\nabla_xF(s,\phi^h(s))]ds\\
&\ \ \ -\frac{1}{2}\int_0^1\langle B^{-1}[(A\phi^h(t)+F(t,\phi^h(t))-\dot{\phi}^h(t)],B^{-1}\eta\rangle dt\\
&\ \ \ +\sum\limits_{0\leq t\leq 1}[V(y(t))-V(y(t-))]+O(\epsilon),
\end{split}
\end{equation}
where $\eta=\int_{|\xi|_H<1}\xi\nu(d\xi)$. As for the remaining term, we have the following control
\begin{equation}
\begin{split}
\sum_{0\leq t\leq 1}[V(y(t))-V(y(t-))]\leq |DV|\sum_{0\leq t\leq 1}|y(t)-y(t-)|,
\end{split}
\end{equation}
where the operator norm $|DV|$ is finite since the operator $D_xV(x,t)=(B^{-1})^{*}(B^{-1}F(x,t)-h(t))$ is $C_b^2$ in $x$ uniformly in $t\in[0,1]$. And $y(t)$ is given by
\begin{equation}
\begin{split}
y(t) &= \phi^h(t)+\int_0^t\exp((t-s)A)BdW(s)+\int_0^t\int_{|y|_H<1}\exp((t-s)A)y \tilde{N}(dt,dy)\\
&\ \ +\int_0^t\int_{|y|_H\geq1}\exp((t-s)A)y N(dt,dy).
\end{split}
\end{equation}

Since $\phi^h(t)+\int_0^t\exp((t-s)A)BdW(s)$ is continuous in $t\in [0,1]$, we have
\begin{equation}
\begin{split}
\sum_{0\leq t\leq 1}|y(t)-y(t-)|=\sum_{0\leq t\leq 1}|z(t)-z(t-)|,
\end{split}
\end{equation}
where $z(t)=\int_0^t\int_{|y|_H<1}\exp((t-s)A)y \tilde{N}(dt,dy)+\int_0^t\int_{|y|_H\geq1}\exp((t-s)A)y N(dt,dy)$. Then by lemma 2.4, we obtain
\begin{equation}
\begin{split}
\sum_{0\leq t\leq 1}|z(t)-z(t-)|=\sum\limits_{j=1}^{\infty}\frac{1}{\alpha_j}(1-e^{-\alpha_jt})\int_{|z|<1}|z|\nu_j(dz)<\infty.
\end{split}
\end{equation}
Combining (3.8), (3.9), (3.27) and (3.32), we have
\begin{equation}
\begin{split}
&\mathbf{P}(\|X-\phi^h\|_2\leq\varepsilon)\propto\mathbf{E}[\exp\{J_0(\phi^h,\dot{\phi}^h)-C+O(\epsilon)\}\mathbf{1}_{\|X^A\|_2\leq\varepsilon}]
\end{split}
\end{equation}
where $C=\sum\limits_{j=1}^{\infty}\frac{1}{\alpha_j}(1-e^{-\alpha_jt})\int_{|z|<1}|z|\nu_j(dz)$ and  $J_0(\phi^h,\dot{\phi}^h)$ is given by
\begin{equation}
\begin{split}
J_0(\phi^h,\dot{\phi}^h)&=-\frac{1}{2}\int_0^1|B^{-1}[(A\phi^h(t)+F(t,\phi^h(t))-\dot{\phi}^h(t)]|_H^2dt\\
&\ \ \ -\frac{1}{2}\int_0^1Tr[\nabla_xF(t,\phi^h(t))]dt\\
&\ \ \ -\frac{1}{2}\int_0^1\langle B^{-1}[(A\phi^h(t)+F(t,\phi^h(t))-\dot{\phi}^h(t)],B^{-1}\eta\rangle dt.
\end{split}
\end{equation}
Hence, we obtain
\begin{equation}
\begin{split}
\lim\limits_{\epsilon\rightarrow0}\gamma_{\epsilon}(\phi)=\lim\limits_{\epsilon\rightarrow0}\frac{P(\|X-\phi^h\|_2\leq\epsilon)}{P(\|X^A\|_2\leq\epsilon)}\propto \exp\{J_0(\phi^h,\dot{\phi}^h)\}.
\end{split}
\end{equation}
Thus, $J_0(\phi^h,\dot{\phi}^h)$ is the desired Onsager-Machlup action functional by the definition in section 2.3. And up to a constant, it can also be written as
\begin{equation}
\begin{split}
J_0(\phi^h,\dot{\phi}^h)&=-\frac{1}{2}\int_0^1|B^{-1}[(A\phi^h(t)+F(t,\phi^h(t))-\dot{\phi}^h(t)-\eta]|_H^2dt\\
&\ \ \ -\frac{1}{2}\int_0^1Tr[\nabla_xF(s,\phi^h(s))]ds,
\end{split}
\end{equation}
where $\eta=\int_{|\xi|_H<1}\xi\nu(d\xi)$. $\square$
\end{proof}

\section{An example}
In this section, we present an example to illustrate the Onsager-Machlup action functional.
\begin{example}
Suppose that $H=L^2([0,1])$ with Neumann boundary conditions, that is $u'(0)=u'(1)=0, A=\triangle, B=Id$, and $F:[0,1]\times L^2([0,1];H)\rightarrow L^2([0,1];H)$ is given by
\begin{equation}
\begin{split}
[F(t,\xi)](t,x)=f(0)\chi_{[0,1]}(t)\chi(x)+f(\int_0^{\frac{1}{2}}dv\int_0^1dy\xi(v,y))\chi_{[\frac{1}{2},1](t)}\chi(x)
\end{split}
\end{equation}
where $f:\mathbb{R}\rightarrow \mathbb{R}$ is a $C_b^2$ function. In addition, the L\'{e}vy process is given by (2.3) with one sided exponentially tempered stable L\'{e}vy process in Example 1. Then the conditions of in Remark 3.9 are satisfied. The Onsager-Machlup action functional for the stochastic partial differential equation $dX(t)=[\Delta X(t)+F(t,X(t))]dt+dW(t)+dL(t)$ is given by
\begin{align*}
J_0(\phi^h,\dot{\phi}^h)&=-\frac{1}{2}\int_0^1|\Delta\phi^h(t)+F(t,\phi^h(t))-\dot{\phi}^h(t)-\eta|_H^2dt,
\end{align*}
where $\eta=\sum_je_j\int_{-1}^1\frac{c_je^{-\beta_jx}}{x^{1+\alpha_j}}dx$.
\end{example}

In fact, for each $\xi\in L^2([0,1];H)$, define
\begin{align*}
l(\xi)=\int_0^{\frac{1}{2}}ds\int_0^1\xi(s,t)dt,
\end{align*}
which defines a linear functional $l$ on $L^2([0,1];H)$. Hence
\begin{align*}
[F(t,\xi)](t,x)=f(0)\chi_{[0,1]}(t)\chi(x)+f(l(\xi))\chi_{[\frac{1}{2},1]}(t)\chi(x).
\end{align*}

Since $f:\mathbb{R}\rightarrow \mathbb{R}$ is a $C_b^2$ function, we obtain that $F(t,\cdot)$ is a $C_b^2(L^2[0,1];H)$ uniformly in $t\in[0,1]$. Moreover, for each $h\in L^2([0,1];H)$, the Frech\'{e}t derivative of $F(t,\cdot)$ is given by
\begin{align*}
DF(t,\xi)(h)&=\lim\limits_{\tau\rightarrow0}\frac{F(t,\xi+\tau h)-F(t,\xi)}{\tau}\\
&=f'(l(\xi))l(h)\chi_{[\frac{1}{2},1]}(t)\chi(x).
\end{align*}
Thus the conditions of in Theorem 3.6 Remark 3.9 are satisfied. And we have $e_j(y)=\cos(2\pi jy)$ for all $y\in[0,1]$ since $H=L^2[0,1]$. Hence, ${e_j\otimes e_k}$ is the complete orthonormal basis of $L^2([0,1],H)$. Then, we obtain
\begin{align*}
Tr[\nabla_xF(s,\phi^h(s))]&=\sum\limits_{j=1}^{\infty}\sum\limits_{k=1}^{\infty}\langle \nabla_xF(s,\phi^h(s))e_j\otimes e_k,e_j\otimes e_k\rangle_{L^2([0,1];H)}\\
&=\sum\limits_{j=1}^{\infty}\sum\limits_{k=1}^{\infty}f'(l(\phi^h))\int_0^{\frac{1}{2}}\int_0^1cos(2\pi jy)\cos(2\pi kz)dydz\\
&\ \ \ \ \ \ \ \ \ \ \ \  \int_{\frac{1}{2}}^1\int_0^1\cos(2\pi jy)cos(2\pi kz)dydz\\
&=0.
\end{align*}
Therefore, the Onsager-Machlup action functional is
\begin{align*}
J_0(\phi^h,\dot{\phi}^h)&=-\frac{1}{2}\int_0^1|\Delta\phi^h(t)+F(t,\phi^h(t))-\dot{\phi}^h(t)-\eta|_H^2dt,
\end{align*}
where $\eta=\sum_je_j\int_{-1}^1\frac{c_je^{-\beta_jx}}{x^{1+\alpha_j}}dx$.

\section{Conclusion and discussion}
In this paper, we derive the Onsager-Machlup action functional for stochastic partial
differential equations with L\'{e}vy process. As far as we know, it is the first attempt for infinite systems with non-Gaussion L\'{e}vy process. Moreover, our derivation for the Onsager-Machlup action functional is available for the L\'{e}vy process with both small and large jumps, while Chao and Duan \cite{chao2019onsager} only dealt with small jumps. Furthermore, we emphasize that the path representing method is an effective attempt for L\'{e}vy process while the $L^2$ techniques will encouter a new difficulty, that is, the representation of the solution process of linear system by Karhunen-Lo\`{e}ve expansion can not result in the independence of random variables.

With the  Onsager-Machlup action functional, it  becomes possible to examine the most probable pathway for systems modeled by stochastic partial differential equations with non-Gaussian L\'evy noise, by studying the associated Euler-Lagrange equations \cite{chao2019onsager, durr1978onsager}.

However, as the similar problem occurs in \cite{chao2019onsager}, we can only deal with the noise having both Brownian motion and L\'{e}vy process. In our methods, we use the Brownian motion to absorb the drift as the Girsanov theorem permits. It is still open to derive the Onsager-Machlup action functional for stochastic dynamical systems with pure L\'{e}vy process both in finite and infinite dimensional systems, because the Girsanov  theorem fails to normalize two probability measures. Bardina et.al \cite{bardina2002asymptotic} tried to deal with jump functions directly rather than using the Girsanov  theorem to obtain the asymptotic evaluation of the Poisson measure for a tube. But L\'{e}vy process is much more complex than Poisson process. A different approach may be needed to attack this open problem.

\section*{Acknowledgements}
We would like to thank Prof. Renming Song for helpful talks and thank Ying Chao, Yanjie Zhang, Qi Zhang, Jinayu Chen, Pingyuan Wei, Yuanfei Huang, Qiao Huang, Wei Wei and Ao Zhang for helpful discussions. This work was partly supported by NSFC grants 11771449 and 11531006.

\bibliographystyle{abbrv}

\bibliography{Refrences}

\begin{thebibliography}{10}

\bibitem{applebaum2009levy}
D.~Applebaum.
\newblock {\em L{\'e}vy processes and stochastic calculus}.
\newblock Cambridge University Press, 2009.

\bibitem{arnold2013random}
L.~Arnold.
\newblock {\em Random dynamical systems}.
\newblock Springer Science \& Business Media, 2013.

\bibitem{bardina2002asymptotic}
X.~Bardina, C.~Rovira, and S.~Tindel.
\newblock Asymptotic evaluation of the {P}oisson measures for tubes around jump
  curves.
\newblock {\em Applicationes Mathematicae}, 29:145--156, 2002.

\bibitem{bardina2003onsager}
X.~Bardina, C.~Rovira, and S.~Tindel.
\newblock Onsager-{M}achlup functional for stochastic evolution equations.
\newblock In {\em Annales de l'Institut Henri Poincare (B) Probability and
  Statistics}, volume~39, pages 69--93. Elsevier, 2003.

\bibitem{bottcher2013levy}
B.~B{\"o}ttcher, R.~Schilling, and J.~Wang.
\newblock L{\'e}vy matters. iii.
\newblock {\em Lecture Notes in Mathematics}, 2099, 2013.

\bibitem{bouchet2014langevin}
F.~Bouchet, J.~Laurie, and O.~Zaboronski.
\newblock Langevin dynamics, large deviations and instantons for the
  quasi-geostrophic model and two-dimensional {E}uler equations.
\newblock {\em Journal of Statistical Physics}, 156(6):1066--1092, 2014.

\bibitem{brzezniak2014strong}
Z.~Brze{\'z}niak, W.~Liu, and J.~Zhu.
\newblock Strong solutions for {SPDE} with locally monotone coefficients driven
  by {L}{\'e}vy noise.
\newblock {\em Nonlinear Analysis: Real World Applications}, 17:283--310, 2014.

\bibitem{capitaine1995onsager}
M.~Capitaine.
\newblock Onsager-{M}achlup functional for some smooth norms on {W}iener space.
\newblock {\em Probability theory and related fields}, 102(2):189--201, 1995.

\bibitem{capitaine2000onsager}
M.~Capitaine.
\newblock On the {O}nsager-{M}achlup functional for elliptic diffusion
  processes.
\newblock In {\em S{\'e}minaire de Probabilit{\'e}s XXXIV}, pages 313--328.
  Springer, 2000.

\bibitem{chao2019onsager}
Y.~Chao and J.~Duan.
\newblock The {O}nsager-{M}achlup function as {L}agrangian for the most
  probable path of a jump-diffusion process.
\newblock {\em Nonlinearity}, 32(10):3715, 2019.

\bibitem{cont1975financial}
R.~Cont and P.~Tankov.
\newblock Financial modelling with jump processes, 2004.
\newblock {\em Chapman \&amp}, 1975.

\bibitem{cooper2016auxiliary}
F.~Cooper and J.~F. Dawson.
\newblock Auxiliary field loop expansion of the effective action for a class of
  stochastic partial differential equations.
\newblock {\em Annals of Physics}, 365:118--154, 2016.

\bibitem{da2014stochastic}
G.~Da~Prato and J.~Zabczyk.
\newblock {\em Stochastic equations in infinite dimensions}.
\newblock Cambridge University Press, 2014.

\bibitem{dembo1991onsager}
A.~Dembo and O.~Zeitouni.
\newblock Onsager-{M}achlup functionals and maximum a posteriori estimation for
  a class of non-{G}aussian random fields.
\newblock {\em Journal of multivariate analysis}, 36(2):243--262, 1991.

\bibitem{ditlevsen1999observation}
P.~D. Ditlevsen.
\newblock Observation of $\alpha$-stable noise induced millennial climate
  changes from an ice-core record.
\newblock {\em Geophysical Research Letters}, 26(10):1441--1444, 1999.

\bibitem{duan2015introduction}
J.~Duan.
\newblock {\em An introduction to stochastic dynamics}.
\newblock Cambridge University Press, 2015.

\bibitem{duan2014effective}
J.~Duan and W.~Wei.
\newblock {\em Effective dynamics of stochastic partial differential
  equations}.
\newblock Elsevier, 2014.

\bibitem{durr1978onsager}
D.~D{\"u}rr and A.~Bach.
\newblock The {O}nsager-{M}achlup function as {L}agrangian for the most
  probable path of a diffusion process.
\newblock {\em Communications in Mathematical Physics}, 60(2):153--170, 1978.

\bibitem{freidlin2012random}
M.~I. Freidlin and A.~D. Wentzell.
\newblock Random perturbations of {H}amiltonian systems.
\newblock In {\em Random Perturbations of Dynamical Systems}, pages 258--354.
  Springer, 2012.

\bibitem{fujita1982onsager}
T.~Fujita and S.-i. Kotani.
\newblock The {O}nsager-{M}achlup function for diffusion processes.
\newblock {\em Journal of Mathematics of Kyoto University}, 22(1):115--130,
  1982.

\bibitem{hao2014asymmetric}
M.~Hao, J.~Duan, R.~Song, and W.~Xu.
\newblock Asymmetric non-{G}aussian effects in a tumor growth model with
  immunization.
\newblock {\em Applied Mathematical Modelling}, 38(17-18):4428--4444, 2014.

\bibitem{hara2016stochastic}
K.~Hara and Y.~Takahashi.
\newblock Stochastic analysis in a tubular neighborhood or {O}nsager-{M}achlup
  functions revisited.
\newblock {\em arXiv preprint arXiv:1610.06670}, 2016.

\bibitem{hochberg1999effective}
D.~Hochberg, C.~Molina-Paris, J.~Perez-Mercader, and M.~Visser.
\newblock Effective action for stochastic partial differential equations.
\newblock {\em Physical Review E}, 60(6):6343, 1999.

\bibitem{huang2020estimating}
Y.~Huang, Y.~Chao, W.~Wei, and J.~Duan.
\newblock Estimating the most probable transition time for stochastic dynamical
  systems.
\newblock {\em arXiv:2006.10979}, 2020.

\bibitem{ikeda2014stochastic}
N.~Ikeda and S.~Watanabe.
\newblock {\em Stochastic differential equations and diffusion processes}.
\newblock Elsevier, 2014.

\bibitem{imkeller2012stochastic}
P.~Imkeller and J.-S. Von~Storch.
\newblock {\em Stochastic climate models}, volume~49.
\newblock Birkh{\"a}user, 2012.

\bibitem{jacod2013limit}
J.~Jacod and A.~N. Shiryaev.
\newblock {\em Limit theorems for stochastic processes}.
\newblock Springer Science \& Business Media, 2013.

\bibitem{jourdain2012levy}
B.~Jourdain, S.~M{\'e}l{\'e}ard, and W.~A. Woyczynski.
\newblock L{\'e}vy flights in evolutionary ecology.
\newblock {\em Journal of mathematical biology}, 65(4):677--707, 2012.

\bibitem{kriegl1997convenient}
A.~Kriegl and P.~W. Michor.
\newblock {\em The convenient setting of global analysis}, volume~53.
\newblock American Mathematical Soc., 1997.

\bibitem{lang2012fundamentals}
S.~Lang.
\newblock {\em Fundamentals of differential geometry}, volume 191.
\newblock Springer Science \& Business Media, 2012.

\bibitem{li2020machine}
Y.~Li, J.~Duan, and X.~Liu.
\newblock A machine learning framework for computing the most probable paths of
  stochastic dynamical systems.
\newblock {\em arXiv:2010.04114}, 2020.

\bibitem{metivier1982semimartingales}
M.~Metivier.
\newblock Semimartingales: A course on stochastic processes (de gruyter, berlin
  \& new york).
\newblock 1982.

\bibitem{onsager1953fluctuations}
L.~Onsager and S.~Machlup.
\newblock Fluctuations and irreversible processes.
\newblock {\em Physical Review}, 91(6):1505, 1953.

\bibitem{peszat2007stochastic}
S.~Peszat and J.~Zabczyk.
\newblock {\em Stochastic partial differential equations with {L}{\'e}vy noise:
  {A}n evolution equation approach}.
\newblock Cambridge University Press, 2007.

\bibitem{qiao2017path}
H.~Qiao and J.~Wu.
\newblock On the path-independence of the {G}irsanov transformation for
  stochastic evolution equations with jumps in {H}ilbert spaces.
\newblock {\em arXiv:1707.07828}, 2017.

\bibitem{qiu2000kuroshio}
B.~Qiu and W.~Miao.
\newblock Kuroshio path variations south of {J}apan: {B}imodality as a
  self-sustained internal oscillation.
\newblock {\em Journal of Physical Oceanography}, 30(8):2124--2137, 2000.

\bibitem{ren2020identifying}
J.~Ren and J.~Duan.
\newblock Identifying stochastic governing equations from data of the most
  probable transition trajectories.
\newblock {\em arXiv:2002.10251}, 2020.

\bibitem{shepp1992note}
L.~A. Shepp and O.~Zeitouni.
\newblock A note on conditional exponential moments and {O}nsager-{M}achlup
  functionals.
\newblock {\em The Annals of Probability}, pages 652--654, 1992.

\bibitem{sun2019pathwise}
X.~Sun, L.~Xie, and Y.~Xie.
\newblock Pathwise uniqueness for a class of {SPDE}s driven by cylindrical
  $\alpha$-stable processes.
\newblock {\em Potential Analysis}, pages 1--17, 2019.

\bibitem{takahashi1981probability}
Y.~Takahashi and S.~Watanabe.
\newblock The probability functionals ({O}nsager-{M}achlup functions) of
  diffusion processes.
\newblock In {\em Stochastic Integrals}, pages 433--463. Springer, 1981.

\bibitem{wolf1993onsager}
E.~M. Wolf and O.~Zeitouni.
\newblock Onsager-{M}achlup functionals for non trace class {SPDE}'s.
\newblock {\em Probability theory and related fields}, 95(2):199--216, 1993.

\bibitem{zhang2019transition}
X.~Zhang, B.~Yang, C.~Wei, and M.~Luo.
\newblock The transition of energy and bound states in the continuum of
  fractional {S}chr{\"o}dinger equation in gravitational field and the effect
  of the minimal length.
\newblock {\em Communications in Nonlinear Science and Numerical Simulation},
  67:290--302, 2019.

\bibitem{zheng2016transitions}
Y.~Zheng, L.~Serdukova, J.~Duan, and J.~Kurths.
\newblock Transitions in a genetic transcriptional regulatory system under
  {L}{\'e}vy motion.
\newblock {\em Scientific reports}, 6:29274, 2016.

\bibitem{zheng2020maximum}
Y.~Zheng, F.~Yang, J.~Duan, X.~Sun, L.~Fu, and J.~Kurths.
\newblock The maximum likelihood climate change for global warming under the
  influence of greenhouse effect and {L}{\'e}vy noise.
\newblock {\em Chaos: An Interdisciplinary Journal of Nonlinear Science},
  30(1):013132, 2020.

\end{thebibliography}

\end{document}